\documentclass[a4paper,11pt]{article}
\usepackage[utf8]{inputenc}

\usepackage{amsmath}
\usepackage{amssymb}
\usepackage{amsfonts}
\usepackage[retainorgcmds]{IEEEtrantools} 
\usepackage{verbatim}
\usepackage{amsthm}
\usepackage{xcolor,graphicx}
\usepackage{mathrsfs}
\usepackage[bookmarks]{hyperref}

\newcommand{\frz}{\mathfrak{z}}
\DeclareMathOperator{\cl}{span}

\DeclareMathOperator{\orb}{orb}
\DeclareMathOperator{\tr}{tr}
\DeclareMathOperator{\diag}{diag}

\newcommand{\bbR}{\mathbb{R}}
\newcommand{\bbD}{\mathbb{D}}

\newcommand{\C}{\mathbb{C}}

\newcommand{\Z}{\mathbb{Z}}

\newcommand{\te}{\tilde{e}}

\newcommand{\fw}{{\mathfrak{w}}}

\newcommand{\fv}{{\mathfrak{v}}}
\newcommand{\cJ}{{\mathcal{J}}}

\theoremstyle{definition}	\newtheorem{defreflectionless}{Definition}[section]
				\newtheorem{deffingapclass}[defreflectionless]{Definition}
\theoremstyle{definition}	\newtheorem{defhtwo}[defreflectionless]{Definition}
\theoremstyle{definition}	
\theoremstyle{plain}		\newtheorem{thmonb}[defreflectionless]{Theorem}
				\newtheorem{thmmultbyz}[defreflectionless]{Theorem}
\theoremstyle{plain}		\newtheorem{thmonbsmp}[defreflectionless]{Theorem}
				\newtheorem{thmmultbyzsmp}[defreflectionless]{Theorem}
				\newtheorem{thmjacobiflowper}[defreflectionless]{Theorem}

\theoremstyle{plain}		
\theoremstyle{plain}		

				\newtheorem{thmunitriples}[defreflectionless]{Theorem}

\theoremstyle{definition}	

\theoremstyle{definition}	

\theoremstyle{plain}

\theoremstyle{plain}		

\theoremstyle{plain}		\newtheorem{thrmekt}[defreflectionless]{Theorem}

\newtheorem{theorem}{Theorem}[section]

\newtheorem{lemma}[defreflectionless]{Lemma}
\newtheorem{proposition}[defreflectionless]{Proposition}
\newtheorem{corollary}[defreflectionless]{Corollary}

\theoremstyle{definition}
\newtheorem{definition}[defreflectionless]{Definition}

\newtheorem{remark}[defreflectionless]{Remark}

\title{A Functional Model for SMP Matrices and the Jacobi Flow}
\author{Benjamin Eichinger, Florian Puchhammer, Peter Yuditskii\thanks{Supported by the Austrian Science Fund FWF, project no: P22025-N18.}}

\begin{document}

\maketitle

 \begin{abstract}
This is the second part of the paper \cite{kssmp} on the theory of SMP (Strong Moment Problem) matrices and their relation
to the Killip-Simon problem on two disjoint intervals. In this part we define and study the Jacobi flow on SMP matrices.
 \end{abstract}
This is the second part of the paper \cite{kssmp} on the theory of SMP (Strong Moment Problem) matrices and their relation
to the Killip-Simon problem on two disjoint intervals. In this part we define and study the Jacobi flow on SMP matrices. Besides
the list of references given in \cite{kssmp}, we would like to mention that an extensive bibliography of works concerned with
the strong moment problem can be found in the survey \cite{JN}; concerning its matrix generalization see \cite{Sim1,Sim2}.

The authors are thankful to Robert Ensgraber for technical support.

\section{Functional Models for Periodic SMP and \\ Isospectral Almost Periodic Jacobi Matrices}
\subsection{Finte-gap Jacobi Matrices: Functional Model}
We start with a formal definition of the class of \emph{reflectionless Jacobi matrices}. Let $J$ be a two-sided Jacobi matrix acting
on $\ell^2$
\begin{equation*}
 Je_n=a_n e_{n-1}+b_ne_n+a_{n+1}e_{n+1},
\end{equation*}
where the $e_n$'s are the vectors of the standard basis. Let $\ell^{2}_{+}=\cl_{n\geq0}\{e_n\}$ and 
$\ell^{2}_{-}=\cl_{n\leq-1}\{e_n\}$. Throughout the paper we assume that the corresponding operators are bounded. We define
\begin{equation*}
 J_{\pm}=P_{\pm}J\vert\ell^{2}_{\pm},
\end{equation*}
and the \emph{resolvent functions}
\begin{equation}\label{eqrezf}
 r_{\pm}(z)=\left\langle(J_{\pm}-z)^{-1}e_{\frac{-1\pm1}{2}},e_{\frac{-1\pm1}{2}}\right\rangle.
\end{equation}
\begin{defreflectionless}
 \label{def:reflectionless}
A Jacobi matrix $J$ is called reflectionless on a Borel set $A\subset\mathbb R$ of positive Lebesgue measure, $|A|>0$, if
\begin{equation*}
 \frac{1}{r_+(x+i0)}=\overline{a_0^2r_{-}(x+i0)}\quad\text{for almost all }x\in A.
\end{equation*}

\end{defreflectionless}


The following isospectral sets of operators are classical in spectral theory and in the theory of integrable systems, see e.g.
\cite[pp. 781-782]{SimOPUC}.

\begin{deffingapclass}
 Let $E$ be a system of $g+1$ proper closed intervals, i.e.
\begin{equation*}
 E=[\mathbf b_0,\mathbf a_0]\setminus\bigcup_{j=1}^g(\mathbf a_j,\mathbf b_j).
\end{equation*}
The finite-gap class $J(E)$ is formed by Jacobi matrices $J$ which are reflectionless on their spectral set $E$.
\end{deffingapclass}

In what follows we will use functional models for matrices of this class in the form as considered in \cite{sodyud}. To this end we need
to define certain special functions related to function theory in
 the common resolvent domain $\Omega=\overline{\mathbb C}\setminus E$ for $J\in J(E)$.

Let $\mathbb D/\Gamma\simeq\overline{\mathbb C}\setminus E$ be a uniformization of the domain $\Omega$. It means that there exists a Fuchsian
group $\Gamma$ and a meromorpic function $\frz:\mathbb D\rightarrow\overline{\mathbb C}\setminus E$, $\frz\circ\gamma=\frz$ for all
$\gamma\in\Gamma$, such that
\begin{equation*}
 \forall z\in\overline{\mathbb C}\setminus E~\exists\zeta\in\mathbb D\!:~\frz(\zeta)=z \text{ and } \frz(\zeta_1)=\frz(\zeta_2)\Rightarrow
\zeta_1=\gamma(\zeta_2).
\end{equation*}
We assume that $\frz$ meets the normalization $\frz(0)=\infty$, $(\zeta \frz)(0)>0$.

Let $\Gamma^{*}$ be the group of characters of the discrete group $\Gamma$.

\begin{defhtwo}
 \label{def:htwo}
For $\alpha\in\Gamma^*$
we define the Hardy  space of character automorphic functions as
\begin{equation*}
H^2(\alpha) = H^2_{\Omega}(\alpha) = \{ f \in H^2\!:~ f \circ \gamma = \alpha(\gamma) f,~\gamma\in\Gamma \},
\end{equation*}
where $H^2$ denotes the standard Hardy class in $\mathbb D$.
\end{defhtwo}

Fix $z_0\in\Omega$ and let $\orb(\zeta_0)=\frz^{-1}(z_0)=\{\gamma(\zeta_0)\}_{\gamma\in\Gamma}$. The Blaschke product $b_{z_0}$ with zeros 
$\frz^{-1}(z_0)$ is called the Green function of the group $\Gamma$ (cf.~\cite{sodyud}). It is related to the standard Green
function $G(z,z_0)$ in the domain $\Omega$ by
\begin{equation*}
\log \frac{1}{|b_{z_0}(\zeta)|} = G\left(\mathfrak{z}(\zeta),z_0\right).
\end{equation*}
The function $b_{z_0}$ is character automorphic, that is, $b_{z_0}\circ\gamma=\mu_{z_0}b_{z_0}$, where $\mu_{z_0}\in\Gamma^{*}$.

We define $k_{\zeta_0}^{\alpha}(\zeta)=k^{\alpha}(\zeta,\zeta_0)$ as the reproducing kernel of the space $H^2(\alpha)$, that is,
\begin{equation*}
 \left\langle f, k_{\zeta_0}^{\alpha}\right\rangle = f(\zeta_0)\quad \forall f\in H^2(\alpha).
\end{equation*}

Let $k^{\alpha}(\zeta)=k_{0}^{\alpha}(\zeta)$, $b(\zeta)=b_{\frz(0)}(z)$, and  $\mu=\mu_{\frz(0)}$. We have an evident decomposition
\begin{equation}\label{ort1}
 H^{2}(\alpha)=\{e^{\alpha}\}\oplus b H^2(\alpha\mu^{-1}), \quad e^{\alpha}=\frac{k^{\alpha}(\zeta)}{\sqrt{k^{\alpha}(0)}}.
\end{equation}
This decomposition plays an essential role in the proof of the following theorem.

\begin{thmonb}
 \label{thm:onb}
The system of functions 
\begin{equation*}
 e_{n}^{\alpha}(\zeta)=b^{n}(\zeta)\frac{k^{\alpha\mu^{-n}}(\zeta)}{\sqrt{k^{\alpha \mu^{-n}}(0)}}
\end{equation*}

\begin{itemize}
 \item[(i)] forms an orthonormal basis in $H^2(\alpha)$ for $n\in \mathbb N$ and
 \item[(ii)] forms an orthonormal basis in $L^2(\alpha)$ for $n\in\mathbb Z$,
\end{itemize}
where 
\begin{equation*}
L^2(\alpha) =  \{ f \in L^2\!:~ f \circ \gamma = \alpha(\gamma) f,~\gamma\in\Gamma \}.
\end{equation*}
\end{thmonb}
\begin{proof}
 Item (i) follows from the above paragraphs and a proof for (ii) in a much more general case can be found in \cite[Theorem E]{sodyud}.
\end{proof}

The following theorem describes all elements of $J(E)$ with a given finite gap set $E$.

\begin{thmmultbyz}
 \label{thm:multbyz}
The multiplication operator by $\frz$ in $L^2(\alpha)$ with respect to the basis $\{e_n^{\alpha}\}$ from Theorem~\ref{thm:onb}
is the following Jacobi matrix $J=J(\alpha)$:
\begin{equation*}
 \frz e_{n}^{\alpha}=a_{n}(\alpha)e_{n-1}^{\alpha} + b_n(\alpha)e_{n}^{\alpha}+a_{n+1}(\alpha)e^{\alpha}_{n+1},
\end{equation*}
where
\begin{equation*}
 a_n(\alpha)=\mathcal A(\alpha\mu^{-n}), \quad\mathcal A(\alpha)=(\frz b)(0)\sqrt{\frac{k^{\alpha}(0)}{k^{\alpha\mu}(0)}}
\end{equation*}
and 
\begin{equation*}
 b_n(\alpha)=\mathcal B(\alpha\mu^{-n}),~~ \mathcal B(\alpha)=\frac{\frz b(0)}{b^{\prime}(0)}+
\left\{  \frac{\left(k^{\alpha}\right)^{\prime}(0)}{k^{\alpha}(0)}- \frac{\left(k^{\alpha\mu}\right)^{\prime}(0)}{k^{\alpha\mu}(0)}\right\}
+\frac{\left(\frz b\right)^{\prime}(0)}{b^{\prime}(0)}.
\end{equation*}
This Jacobi matrix $J(\alpha)$ belongs to $J(E)$. Thus, we have a map from $\Gamma^{*}$ to $J(E)$. Moreover, this map is one-to-one.
\end{thmmultbyz}

\begin{remark}
An easy consequence of the above functional model is the following important relation
\begin{equation}\label{sjm}
S^{-1}J(\alpha) S=J(\mu^{-1}\alpha), \quad S e_n:= e_{n+1}.
\end{equation}
In particular, $J(\alpha)$ is periodic if and only if $\mu^N=1_{\Gamma^*}$ for a certain positive integer $N$.
\end{remark}

\subsection{Periodic SMP matrices}

Let $\tilde E$ be a system of two proper closed intervals.  By a linear change of variables we can pass to the system of intervals $E=[\mathbf b_0,\mathbf a_0]\setminus (\mathbf a_1,\mathbf b_1)$ 
such that
\begin{equation}\label{edef}
E=V^{-1}([-2,2])
\end{equation}
with an appropriate function
\begin{equation*}
V(z)=\frak a z+\frak b -\frac 1 z, \quad \frak a>0,\ \frak b\in\bbR.
\end{equation*}

Let further $A(E)$ denote the set of all period-two SMP matrices with their spectrum on $E$.
Recall that in our previous paper \cite{kssmp} we characterized the elements of $A(E)$ by a magic formula $V(A)=S^2+S^{-2}$.

By definition $V\!:\Omega\rightarrow \overline{\mathbb C}\setminus[-2,2]$, as before, $\Omega=\bar\C\setminus E$.
On the other hand, for the classical Joukowski map we have 
 $\Delta+\Delta^{-1}\!:\mathbb D\rightarrow\overline{\mathbb C}\setminus[-2,2]$. Thus, a single-valued function $\Delta(z)$ is uniquely defined in $\Omega$ by
 \begin{equation}\label{defdelta}
 \Delta(z)+\frac 1{\Delta(z)}=V(z), \  |\Delta(z)|<1, \quad z\in\Omega.
\end{equation}
Let us list its characteristic properties
\begin{itemize}
 \item[(i)] $|\Delta|<1$ in $\Omega$ and $|\Delta|=1$ on $E$,
 \item[(ii)] $\Delta(\infty)=\Delta(0)=0$, otherwise $\Delta(z)\not=0$.
\end{itemize}
All this implies that $\Delta(z)$ is given by $\Delta(\frz(\zeta))=b_{0}(\zeta)b_{\infty}(\zeta)$. In particular, $\mu_0=\mu^{-1}$,
recall that $\mu=\mu_\infty$.

\begin{remark}
We want to mention that the uniformization of the domain 
\begin{equation*}
\mathfrak{z} : \mathbb{C_+}/\Gamma \simeq \Omega,\quad \C_+=\{z: \Im z>0\}
\end{equation*}
in the two interval case can be given explicitly \cite[Chapter VIII]{AKHef}.
Namely,  $\Gamma$ is a discrete group of dilations, $\Gamma=\{\gamma_n\!: \gamma_n(w) = \rho^n w,~n \in \mathbb{Z}\}$, where
\begin{equation*}
\rho=\exp \left [\frac{2 \pi \int_{\mathbf a_0}^{\mathbf b_1} \frac{dz}{\sqrt{|(z-\mathbf b_0) (z-\mathbf b_1) (z-\mathbf a_0) (z-\mathbf a_1)|}}}{\int_{\mathbf a_0}^{\mathbf b_0} \frac{dz}{\sqrt{|(z-\mathbf b_0) (z-\mathbf b_1) (z-\mathbf a_0) (z-\mathbf a_1)|}}}\right ],
\end{equation*}
and $z=\frz(w)$ is the inverse of the function $w(z)$ given in the upper half plane by
\begin{equation*}
w=\exp \left [\frac{i \pi \int_{\mathbf a_0}^{z} \frac{dz}{\sqrt{(z-\mathbf b_0) (z-\mathbf b_1) (z-\mathbf a_0) (z-\mathbf a_1)}}}{\int_{\mathbf a_0}^{\mathbf b_0} \frac{dz}{\sqrt{|(z-\mathbf b_0) (z-\mathbf b_1) (z-\mathbf a_0) (z-\mathbf a_1)|}}}\right ]
\end{equation*}
and extended to $\C_+$ of the $w$-plane by the symmetry principle. All corresponding functions as well as the Green functions and the reproducing kernels can be written explicitly by means of elliptic functions.
We will not give these explicit formulas here, however.
\end{remark}

Let us fix $\zeta_*\in\bbD$ such that $\frz(\zeta_*)=0$ and $\gamma_1(\zeta_*)=\bar \zeta_*$ for the generator $\gamma_1$ of the group $\Gamma$. In order to construct a functional model for operators from $A(E)$ we start with the following counterpart of the orthogonal decomposition \eqref{ort1}:
\begin{equation}\label{smpbase0}
H^2(\alpha)=\{k^{\alpha}_{\zeta_*}, k^{\alpha}_0\}\oplus\Delta H^2(\alpha)
=\{f^{\alpha}_0\}\oplus\{f^{\alpha}_1\}\oplus\Delta H^2(\alpha),
\end{equation}
where
\begin{equation}\label{smpbase}
f_0^{\alpha}=\frac{\sqrt{\alpha^{-1}(\gamma_1)}k^\alpha_{\zeta_*}}{\sqrt{k^\alpha_{\zeta_*}(\zeta_*)}}, \quad
f_1^\alpha=\frac{b_{\frz(\zeta_*)} k^{\alpha\mu}}{\sqrt{k^{\alpha\mu}(0)}}.
\end{equation}

\begin{thmonbsmp}
 \label{thm:onbsmp}
The system of functions 
\begin{equation*}
f_{n}^{\alpha}=\begin{cases}\Delta^m f_0^\alpha,& n=2m\\
\Delta^m f_1^\alpha,& n=2m+1
\end{cases}
\end{equation*}
\begin{itemize}
 \item[(i)] forms an orthonormal basis in $H^2(\alpha)$ for $n\in \mathbb N$ and
 \item[(ii)] forms an orthonormal basis in $L^2(\alpha)$ for $n\in\mathbb Z$.
\end{itemize}
\end{thmonbsmp}
\begin{proof}
 Item (i) follows from \eqref{smpbase0} and for (ii) we have to use the description of the orthogonal complement 
 $L^2(\alpha)\ominus H^2(\alpha)$, see \cite{sodyud}.
\end{proof}

Similarly as we had before, this allows us to parametrize all elements of $A(E)$ for a given two-interval set $E$ by the 
characters from $\Gamma^*$.

\begin{thmmultbyzsmp}
 \label{thm:multbyzsmp}
Let $E$ be the union of two closed intervals. Then
the multiplication operator by $\frz$ with respect to the basis $\{f_n^{\alpha}\}$ is an SMP matrix $A(\alpha)\in A(E)$,
 \begin{eqnarray*}
\frz f_{2m}^\alpha=& p_0(\alpha)  f_{2m-1}^\alpha+q_0(\alpha) f_{2m}^\alpha +p_1(\alpha) f_{2m+1}^\alpha\\
\frz f_{2m-1}^\alpha=&r_1(\alpha)  f_{2m-3}^\alpha +p_1(\alpha)  f_{2m-2}^\alpha+q_1(\alpha) f_{2m-1}^\alpha +p_0(\alpha) f_{2m}^\alpha+r_1(\alpha)  f_{2m+1}^\alpha
\end{eqnarray*}
Moreover, this map $\Gamma^*\to A(E)$ is one-to-one up to the identification 
$$
(p_0,p_1)\mapsto (-p_0,-p_1) \ \text{in} \ A(E).
$$
\end{thmmultbyzsmp}

\begin{proof}
 See \cite[Theorem 4.2]{ionela} and also \cite[Section 3.3]{kssmp}. We only need to check that, under the normalization \eqref{smpbase}, $p_0(\alpha)$ is real. We have
 $$
 \overline{p_0(\alpha)}=\langle  f_{-1}^\alpha, \frz f_0^\alpha\rangle=(b\frz)(0)\sqrt{\alpha(\gamma_1)}\frac{k^\alpha(\zeta_*,0)}{\sqrt{k^{\alpha\mu}_{\zeta_*}(\zeta_*)k^\alpha(0)}}.
 $$
 Since $\overline{k^\alpha(\bar \zeta)}=k^\alpha(\zeta)$ we get
 $$
k^\alpha(\zeta_*)= \overline{k^\alpha(\bar\zeta_*)}= \overline{k^\alpha(\gamma_1(\zeta_*))}= 
 \overline{\alpha(\gamma_1)k^\alpha(\zeta_*)}.
 $$
 Therefore, $\sqrt{\alpha(\gamma_1)}k^\alpha(\zeta_*)$ is real. Note that the square root of $\alpha(\gamma_1)$ is defined up to the multiplier $\pm 1$.
\end{proof}

\begin{remark}
It is evident that, in this basis, multiplication by $\Delta$  is the shift $S^2$, $\Delta f_n^{\alpha}=f^\alpha_{n+2}$. Thus, the magic formula for SMP matrices corresponds to the definition \eqref{defdelta} indeed. In particular, $\frak a r_1(\alpha)=1$.
\end{remark}

\subsection{Jacobi flow on periodic SMP matices}
In fact all operators in $J(E)$, as well as in $A(E)$, are unitarily equivalent. But there is a more delicate form of an equivalence relation, which allows to distinguish between all of them.

\begin{definition}\label{deftry}
Let $\mathbf A$ be a self-adjoint operator in a Hilbert space $\mathbf H$, and let the vectors 
$\mathbf e_{-1}, \mathbf e_{0}$ form a cyclic subspace. We say that the triple 
$\{\mathbf H, \mathbf A,\{\mathbf e_{-1}, \mathbf e_{0}\}\}$ is equivalent to another triple $\{\mathbf H', \mathbf A',\{\mathbf e'_{-1}, \mathbf e'_{0}\}\}$ of this sort if there exists a unitary operator $\mathbf U:\mathbf H\to\mathbf H'$ such that
\begin{equation*}
\mathbf e'_{-1}=\mathbf U\mathbf e_{-1}, \ \mathbf e'_{0}=\mathbf U\mathbf e_{0},\ \text{and}\
\mathbf A' \mathbf U=\mathbf U \mathbf A.
\end{equation*}
\end{definition}

\begin{proposition}\label{propuep}
The triples
\begin{equation*}
\{\ell^2, J(\alpha), \{e_{-1}, e_0\}\}\quad\text{and}\quad \{\ell^2, A(\alpha), \{e_{-1}, \tilde e_0(\alpha)\}\},
\end{equation*}
 where 
\begin{equation}\label{tildee0}
 \tilde e_0(\alpha)=\frac{p_0(\alpha)}{\sqrt{p_0^2(\alpha)+r_1^2(\alpha)}}e_0+\frac{r_1(\alpha)}{\sqrt{p_0^2(\alpha)+r_1^2(\alpha)}}e_1,
\end{equation}
are unitarily equivalent.
\end{proposition}

\begin{proof}
We use the functional models given in Theorems \ref{thm:multbyz} and \ref{thm:multbyzsmp}. The first triple is equivalent to
$\{L^2(\alpha),\frz, \{e_{-1}^\alpha, e_{0}^\alpha\}\}$. Note that $e_{-1}^\alpha$ and $f_{-1}^\alpha$ simply coincide. Further, by Theorem  \ref{thm:multbyz}, the projection of $\frz e_{-1}^\alpha$ onto $H^2(\alpha)$ is of the form
$a_0(\alpha) e_0^\alpha$. By Theorem \ref{thm:multbyzsmp} the same vector has the decomposition
\begin{equation*}
p_0(\alpha) f_0^\alpha+r_1(\alpha)f_1^\alpha.
\end{equation*}
Since $a_0(\alpha)>0$ we have
\begin{equation}\label{tildee}
e_0^\alpha=\frac{p_0(\alpha) f_0^\alpha+r_1(\alpha)f_1^\alpha}{a_0(\alpha)}, \quad a_0(\alpha)={\sqrt{p_0^2(\alpha)+r_1^2(\alpha)}}.
\end{equation}
Thus, the given triple can be rewritten into the form $\{L^2(\alpha),\frz, \{f_{-1}^\alpha, e_{0}^\alpha\}\}$. By \eqref{tildee}
we get $ \{\ell^2, A(\alpha), \{e_{-1}, \tilde e_0(\alpha)\}\}$ with $\tilde e_0(\alpha)$ as given in \eqref{tildee0}.
\end{proof}

\begin{remark}
Note that we have simultaneously proved that $\{e_{-1}, \tilde e_0(\alpha)\}$ forms a cyclic subspace for $A(\alpha)$. Besides,
\begin{equation}\label{tildee1}
q_1(\alpha)=\langle A(\alpha) e_{-1}, e_{-1}\rangle=\langle J(\alpha) e_{-1}, e_{-1}\rangle=b_{-1}(\alpha)=
\mathcal B(\alpha\mu).
\end{equation}
\end{remark}

Recall, see \eqref{sjm}, that the shift of a Jacobi matrix  $J(\alpha)$ corresponds to the multiplication of the character by $\mu^{-1}$.

\begin{definition}
We define the Jacobi flow  on $A(E)$ as a dynamical system generated by the following map: 
$$
\mathcal J A(\alpha)=A(\mu^{-1}\alpha), \quad \alpha\in\Gamma^*.
$$
\end{definition}

\begin{thmjacobiflowper}
 \label{thm:jacobiflowper}
Let $\mathcal U(\alpha)$ be the periodic $2\times 2$-block diagonal unitary matrix given by
\begin{equation}
\label{eqn:ufunctional}
 \mathcal U(\alpha)\begin{bmatrix}e_{2m}& e_{2m+1}
 \end{bmatrix}=\begin{bmatrix}e_{2m}& e_{2m+1}
 \end{bmatrix}u(\alpha), 
 \end{equation}
 where
 \begin{equation}
\label{eqn:ufunctional1}
 u(\alpha)=\frac{1}{\sqrt{p_0^2(\alpha)+r_1^2(\alpha)}}
 \begin{bmatrix}p_0(\alpha)&r_1(\alpha)\\r_1(\alpha)&-p_0(\alpha) \end{bmatrix}
 \end{equation}
 Then
\begin{equation}\label{jfp}
 \mathcal J A(\alpha):=A(\mu^{-1}\alpha)=S^{-1}\mathcal U(\alpha)^*A(\alpha)\mathcal U(\alpha)S.
\end{equation}
\end{thmjacobiflowper}

\begin{proof}
In the two dimensional space $H^2(\alpha)\ominus\Delta H^2(\alpha)$ we consider two orthonormal systems
\begin{equation}
\label{twoorth}
\left\{\frac{\sqrt{\alpha^{-1}(\gamma_1)}k^\alpha_{\zeta_*}}{\sqrt{k^\alpha_{\zeta_*}(\zeta_*)}}, \frac{ b_{\frz(\zeta_*)}k^{\alpha\mu}}
{\sqrt{k^{\alpha\mu}(0)}}\right\}\ \text{and}\ 
\left\{\frac{k^\alpha}{\sqrt{k^\alpha(0)}}, 
\sqrt{\frac{\mu(\gamma_1)}{\alpha(\gamma_1)}}\frac{ bk^{\alpha\mu^{-1}}_{\zeta_*}}{\sqrt{k^{\alpha\mu^{-1}}_{\zeta_*}(\zeta_*)}}\right\}.
\end{equation}
They are related by a unitary matrix, which, for a moment, we denote by $v(\alpha)^*$. We will show that $v(\alpha)=u(\alpha)$ as it was defined in \eqref{eqn:ufunctional1}.

According to our notation, see 
\eqref{smpbase}, the first pair of vectors is $\{f_0^\alpha, f_1^\alpha\}$. The first vector of the second pair is $e^\alpha$, see \eqref{ort1}, and we have
$$
e^\alpha=\begin{bmatrix} f_0^\alpha & f_1^\alpha\end{bmatrix}v(\alpha)^*
\begin{bmatrix} 1 \\ 0\end{bmatrix}.
$$
Thus, due to \eqref{tildee},
$$
v(\alpha)^*
\begin{bmatrix} 1 \\ 0\end{bmatrix}=\frac{1}{\sqrt{p_0^2(\alpha)+r_1^2(\alpha)}}
 \begin{bmatrix}p_0(\alpha)&r_1(\alpha)\\ r_1(\alpha)&-p_0(\alpha) \end{bmatrix}\begin{bmatrix} 1 \\ 0\end{bmatrix}.
 $$
 Since both matrices here are unitary we have
 $$
v(\alpha)^*
\begin{bmatrix} 0 \\ 1\end{bmatrix}=\frac{1}{\sqrt{p_0^2(\alpha)+r_1^2(\alpha)}}
 \begin{bmatrix}p_0&r_1\\ r_1&-p_0 \end{bmatrix}(\alpha)\begin{bmatrix} 0 \\ 1\end{bmatrix} c
 $$
for a certain $c\in \mathbb T$. However,
$$
\frac{c r_1(\alpha)}{\sqrt{p_0^2(\alpha)+r_1^2(\alpha)}}=
\left\langle \frac{ \sqrt{\mu(\gamma_1)}bk^{\alpha\mu^{-1}}_{\zeta_*}}{\sqrt{k^{\alpha\mu^{-1}}_{\zeta_*}(\zeta_*)}}, 
\frac{k^\alpha_{\zeta_*}}{\sqrt{k^\alpha_{\zeta_*}(\zeta_*)}}\right\rangle=
\sqrt{\mu(\gamma_1)b(\zeta_*)}\sqrt{\frac{k^{\alpha\mu^{-1}}_{\zeta_*}(\zeta_*)}{k^\alpha_{\zeta_*}(\zeta_*)}},
$$
and since $\overline{b(\bar \zeta)}=b(\zeta)$, the last expression is real and, with a suitable choice of the square root of $\mu(\gamma_1)$, it is positive. In this case, $c=1$.

Now we rewrite the second pair from \eqref{twoorth} as $\{bf^{\alpha\mu^{-1}}_{-1}, bf^{\alpha\mu^{-1}}_{0}\}$. This means we have proved
\begin{equation*}
b\begin{bmatrix} f_{-1}^{\alpha\mu^{-1}} & f_0^{\alpha\mu^{-1}}\end{bmatrix}=
\begin{bmatrix} f_0^\alpha & f_1^\alpha\end{bmatrix}u(\alpha)^*.
\end{equation*}
Multiplying the above expression by $\Delta^m$, $m\in\Z$, shows that
\begin{equation*}
b\begin{bmatrix} f_{2m-1}^{\alpha\mu^{-1}} & f_{2m}^{\alpha\mu^{-1}}\end{bmatrix}=
\begin{bmatrix} f_{2m}^\alpha & f_{2m+1}^\alpha\end{bmatrix}u(\alpha)^*.
\end{equation*}
Therefore, the matrices of the multiplication operator by $\frz$ in these two bases are related by \eqref{jfp}.
 \end{proof}

Recall that, in \cite{kssmp}, $A(E)$ was described as a collection of points $(p_0,p_1)$ on the isospectral curve
\begin{equation}\label{isosp}
p_0^2+p_1^2+\frak a^2 p_0^2p_1^2+\frak bp_0p_1=\frac{1}{\frak a}.
\end{equation}

\begin{corollary}
The Jacobi flow on the isospectral curve is given by the following transformation
\begin{equation}\label{jaconiso}
(p_0^{(1)},p_1^{(1)}):=\mathcal J(p_0,p_1)=\left(p_1+\frac{\frak b p_0}{1+\frak a ^2p_0^2}, -p_0\right).
\end{equation}
Consequently, the coefficients of the Jacobi matrix $J\in J(E)$ associated with a fixed point $(p_0,p_1)$ on the isospectral curve are given:
\begin{equation}\label{jaconiso1}
a^2_n=\frac{1}{\frak a^2}+(p_0^{(n)})^2, \  b_{n-1}=-\frac{\frak b}{\frak a}-\frak a p_0^{(n)} p_1^{(n)},
\quad (p_0^{(n)},p_1^{(n)})=\mathcal J^{\circ n}(p_0,p_1).
\end{equation}

\end{corollary}

\begin{proof}
Let $A^{(1)}=\mathcal J A$, $A=A(\alpha)\in A(E)$. Due to \eqref{jfp} we have
\begin{equation*}
\begin{bmatrix} r_{1}^{(1)} & 0 \\ p_{1}^{(1)} & 0\end{bmatrix}=
 u(\alpha)^*\begin{bmatrix} 0 & 0 \\ p_{0} & r_{1}\end{bmatrix}u(\alpha)
\end{equation*}
and
\begin{equation*}
 \begin{bmatrix} q_{1}^{(1)} & p_{0}^{(1)} \\ p_{0}^{(1)} & q_{0}^{(1)}\end{bmatrix}
 =
 u(\alpha)^*\begin{bmatrix} q_{0} & p_{1} \\ p_{1} & q_{1}\end{bmatrix}u(\alpha).
\end{equation*}
These identities imply \eqref{jaconiso}. For $J=J(\alpha)$ we obtain \eqref{jaconiso1} from \eqref{tildee} and \eqref{tildee1}.
\end{proof}

\begin{figure}[htbp]
\begin{center}
\includegraphics[scale=0.5]
{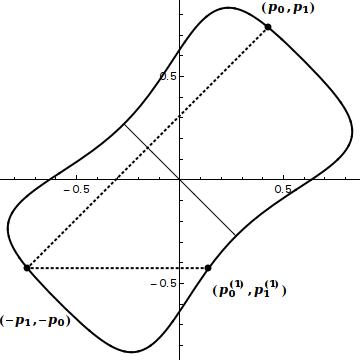}
\caption{Jacobi flow mapping in $(p_0,p_1)$-plane}
\label{jffig}
\end{center}
\end{figure}

\begin{remark}
The transformation \eqref{jaconiso} has the following geometrical meaning: first we pass to the symmetric point 
$(p_0,p_1)\mapsto (-p_1,-p_0)$ and then we are looking for the second point of intersection of the horizontal line, going trough this point, with the isospectral curve, see Fig \ref{jffig}. In Fig. \ref{perfig} we give an example where the parameters $(\frak a,\frak b)$ are chosen in such a way that the trajectory given by the Jacobi flow is periodic (consequently, the Jacobi matrix $J$ is periodic). 
The second choice of parameters, see Fig. \ref{nperfig}, corresponds to almost periodic $J$.
\end{remark}

\begin{figure}[htbp]
\begin{center}
\includegraphics[scale=0.5]
{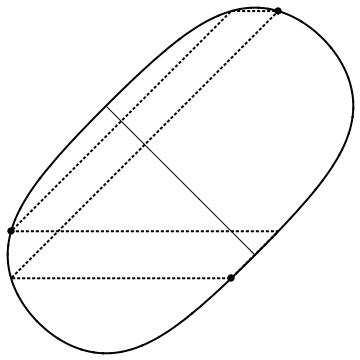}
\caption{$J\in J(E)$ are periodic}
\label{perfig}
\end{center}
\end{figure}

\begin{figure}[htbp]
\begin{center}
\includegraphics[scale=0.5]
{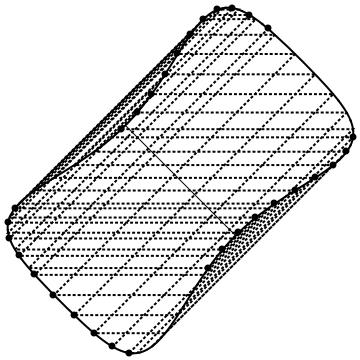}
\caption{$J\in J(E)$ are almost periodic (the trajectory is not closed)}
\label{nperfig}
\end{center}
\end{figure}
\section{Jacobi flow on SMP matrices: Generic Case}

\subsection{Definition of the Jacobi flow}
First of all, following the result from Proposition \ref{propuep} we will establish a connection between Jacobi- and SMP matrices. 
 To this end, we  give a specific cyclic subspace for SMP matrices.
 
 \begin{proposition}
 \label{cor:cyclic}
 Let $A$ be an SMP matrix with defining parameters 
 $$
 \{p_k,q_{2k+1},r_{2k+1}\}_{k\in\Z}
 $$ 
(cf. \cite[Theorem 3.2.2]{kssmp}). To this matrix we associate the vector
\begin{equation}\label{tildee2}
 \te_0=\frac 1{a_0}(p_0e_0+r_1e_1), \quad a_0:=\sqrt{p_0^2+r_1^2}.
\end{equation}
Then, the vectors $e_{-1}$ and $\te_0$  form a cyclic subspace for $A$.
\end{proposition}

We begin with the following lemma.

\begin{lemma}
\label{thm:cyclic}
 Let $A$ be an SMP matrix. Then one of the following holds true.
\begin{enumerate}
 \item $\{e_{-1},e_{0}\}$ forms a cyclic subspace for $A$ (if $\pi_{-1}\neq0$).
 \item $\{e_{-2},e_{-1}\}$ forms a cylclic subspace for $A$ (if $\pi_{-1}=0$).
\end{enumerate}
\end{lemma}

\begin{proof}
Recall that the entries $\pi_{-1}$ and $\pi_0$ of the inverse of an SMP matrix cannot equal zero 
simultaneously, as otherwise, by \cite[Theorem~3.2.2]{kssmp},
 we would have that $\sigma_{-1}=0$. Thus, it would contain a whole line of zeros, which violates
the definition of SMP matrices.

First, assume $\pi_{-1}\neq0$ and consider the equation
\begin{equation*}
 A^{-1}e_{-1}=\pi_{-1}e_{-2} + \sigma_{-1}e_{-1}+\pi_0e_0.
\end{equation*}

Thus, $e_{-2}\in\cl\{A^ne_{-1},A^ne_{0}\}$. Furthermore,
\begin{equation}
\label{eqn:aezero}
 Ae_{0}=p_{0}e_{-1} + q_0e_{0}+p_1e_1
\end{equation}
and hence, note that $p_{1}\neq0$,  $e_{1}\in\cl\{A^ne_{-1},A^n e_{0}\}$.
Given the relations
\begin{IEEEeqnarray}{rCl}
 \label{eqn:ainveeven}
\nonumber A^{-1}e_{2n} &=& \rho_{2n}e_{2n-2}+\pi_{2n}e_{2n-1} \\
&+& \sigma_{2n}e_{2n}+\pi_{2n+1}e_{2n+1}+\rho_{2n+2}e_{2n+2}\\
\label{eqn:aeodd}
\nonumber Ae_{2n+1} &=& r_{2n+1}e_{2n-1}+ p_{2n+1}e_{2n}\\
&+&q_{2n+1}e_{2n+1}+p_{2n+2}e_{2n+2}+r_{2n+3}e_{2n+3}
\end{IEEEeqnarray}
and keeping in mind that $\rho_{2n}\neq0$ as well as $r_{2n+1}\neq0$, we obtain all unit vectors with positive indices by alternately
iterating through (\ref{eqn:ainveeven}) and (\ref{eqn:aeodd}), starting with $n=0$ and those with negative indices by iterating through 
 (\ref{eqn:aeodd}) and (\ref{eqn:ainveeven}), starting with $n=-1$.

Now, if $\pi_{-1}=0$, i.e. $\pi_0\neq0$, we obtain similarly to the first case that $e_0\in\cl\{A^ne_{-2},A^ne_{-1}\}$. Additionally,
$p_{-2}\neq0$ and
\begin{equation*}
 Ae_{-2}=p_{-2}e_{-3}+q_{-2}e_{-2}+p_{-1}e_{-1}.
\end{equation*} 
Therefore, $e_{-3}\in\cl\{A^ne_{-2},A^ne_{-1}\}$. With these vectors in hand we find that $\{e_1,e_2,\ldots\}\subset\cl\{A^ne_{-2},A^ne_{-1}\}$
considering the expressions $Ae_{-1},A^{-1}e_{0},\ldots$, as well as $\{e_{-4},e_{-5},\ldots\}\subset\cl\{A^ne_{-2},A^ne_{-1}\}$ through
$A^{-1}e_{-2}, Ae_{-3},\ldots$
\end{proof}

\begin{proof}[Proof of Proposition \ref{cor:cyclic}]
 Let $\pi_{-1}\neq0$, i.e. $p_1\neq0$. Then, from (\ref{eqn:aezero}) we get that
\begin{equation*}
 e_1=\frac{1}{p_1}\left(Ae_0-p_0e_{-1}-q_0e_0\right).
\end{equation*}
Together with the definition of $\te_0$ it follows that
\begin{IEEEeqnarray*}{rCl}
 a_0\te_0 &=& r_1e_1+p_0e_0\\
&=& \frac{r_1}{p_1}\left(A+(-q_0+\frac{p_0 p_1}{r_1})I\right)e_0-\frac{p_0r_1}{p_1}e_{-1},
\end{IEEEeqnarray*}
and, since $q_{0}=p_0p_1/r_1$, we finally obtain
\begin{equation*}
 e_0=\frac{p_1a_0}{r_1}A^{-1}\te_0+p_0A^{-1}e_{-1},
\end{equation*}
 implying $\{e_{-1},e_{0}\}\subseteq\cl\{A^ne_{-1},A^{n}\tilde e_0\}$.

Now, assuming $\pi_{-1}=0$, we can deduce the following (note that $A^{-1}$ has to be invertible by definition)
\begin{equation*}
 p_1=q_0=\sigma_{-1}=0 \quad\text{and}\quad p_0\neq0\neq\pi_0.
\end{equation*}
Therefore, first of all, we obtain
\begin{equation*}
 e_0=\frac{1}{\pi_0}A^{-1}e_{-1}\quad\text{as well as}\quad e_1=\frac{1}{r_1}(a_0\te_0-p_0e_0)
\end{equation*}
Moreover, we use $\pi_2=p_0\rho_2/r_1\neq0$ to find that
\begin{equation*}
 e_{2}=\frac{1}{\pi_2}\left(A^{-1}e_1-\pi_1e_0-\sigma_1e_1\right).
\end{equation*}
Thus, by looking at the expansion of $A^{-1}e_0$, we obtain
\begin{equation*}
 e_{-2}=\frac{1}{\rho_0}\left(A^{-1}e_0-\pi_0e_{-1}-\sigma_0e_0-\pi_1e_1-\rho_2e_2\right).
\end{equation*}
As we are in a position to express every vector appearing on the right hand-side in terms of $A^ne_{-1}$ and $A^n\te_{0}$
for suitable $n\in\mathbb Z$, we get that $\{e_{-2},e_{-1}\}\subseteq\cl\{A^ne_{-1},A^n\te_{0}\}$. Thus, in any case, the result
follows from Lemma~\ref{thm:cyclic}.
\end{proof}

This allows us to define a mapping $\mathcal F$ from the class of SMP matrices to the class of Jacobi matrices, i.e., to set the above mentioned connection between SMP- and Jacobi matrices.

\begin{thmunitriples}
\label{thm:unitriples}
For any SMP matrix $A$ there exists a unique Jacobi matrix $J$, $J=\mathcal F A$, such that $(\ell^2, J,\{e_{-1},e_0\})$ and 
$(\ell^2,A,\{e_{-1},\tilde{e}_0\})$ are equivalent triples, see Definition \ref{deftry}.
\end{thmunitriples}

\begin{proof}
With $A_\pm$ we denote $A_\pm=P_\pm A P_\pm$, where $P_\pm$ is the orthogonal projection onto $\ell^2_\pm$. 
In \cite{kssmp} we have already mentioned that $A$ can be written as the orthogonal sum of these operators and a two dimensional 
pertubation, i.e.,
\begin{align}
A=A_+\oplus A_- + a_0(\langle .,\tilde{e}_0\rangle e_{-1} + \langle .,e_{-1}\rangle \tilde{e}_0)\label{Jrep},
\end{align} 
Due to Proposition~\ref{cor:cyclic}, the subspace $\{e_{-1},\tilde{e}_{0}\}$ is cyclic for $A$. The matrix resolvent function for $A$ with respect to the given cyclic subspace is defined by
\[
R^A(z)=
\left[
\begin{matrix}
R^A_{-1,-1}(z)&R^A_{-1,0}(z)\\
R^A_{0,-1}(z)&R^A_{0,0}(z) 
\end{matrix}
\right]
=\mathcal{E}^*(A-z)^{-1}\mathcal{E},
\]
where
\[
\mathcal{E}
\left[
\begin{matrix}
c_{-1}\\
c_0
\end{matrix}
\right]
=
c_{-1}e_{-1}+c_0\tilde{e}_0, \quad (c_{-1}, c_2)\in \C^2.
\]
From \eqref{Jrep} we obtain the following identity
\[
R^A(z)=\left[
\begin{matrix}
(r_-^A)^{-1}&a_0\\
a_0&(r^A_+)^{-1}
\end{matrix}
\right]
^{-1},
\]
where
$$
r_-^A(z)=\langle (A_--z)^{-1}e_{-1}, e_{-1}\rangle, \quad
r_+^A(z)=\langle (A_+-z)^{-1}\tilde e_{0},\tilde  e_{0}\rangle.
$$
It is well-known that a Jacobi matrix $J$ has a representation, similar to \eqref{Jrep}. 

We define one-sided Jacobi matrices $J_\pm$ by the resolvent functions $r^A_\pm(z)$, see \eqref{eqrezf}. Jointly with the constant $a_0$ they form a two-sided Jacobi matrix $J$.
 Hence, given a SMP matrix
 $A$, we can construct a Jacobi matrix $J$ such that
\[
R^J(z)=
\left[
\begin{matrix}
R^A_{-1,-1}(z)&R^A_{-1,0}(z)\\
R^A_{0,-1}(z)&R^A_{0,0}(z) 
\end{matrix}
\right].
\]
In fact, since $\{e_{-1},e_{0}\}$ is 
cyclic for $J$, replacing $J$ with $A$ and  $\tilde{e}_{0}$ with $e_{0}$, the above statements stay true.
 The resolvent functions 
$R(z):=R^A(z)=R^J(z)$ have an integral representation
\[
R(z)=\int\frac{d\sigma}{x-z},
\]
 where $\sigma$ is a $2\times 2$ matrix measure. Thus, by the spectral theorem, $A$ and $J$ are unitarily equivalent to the multiplication by an independent variable on
\[
L^2_{d\sigma}=\left\{ f=
\left[
\begin{matrix}
f_{-1}(x)\\
f_0(x)
\end{matrix}
\right]
:\int f^*d\sigma f<\infty
\right\}.
\] 
Moreover, in this spectral representation of $A$ we  have that
\[ 
e_{-1}\mapsto \left[
\begin{matrix}
1\\
0
\end{matrix}
\right]
,\quad
\tilde{e}_{0}\mapsto\left[
\begin{matrix}
0\\
1
\end{matrix}
\right]
,
\]
as well as for $J$ that
\[ 
e_{-1}\mapsto \left[
\begin{matrix}
1\\
0
\end{matrix}
\right]
,\quad
{e}_{0}\mapsto\left[
\begin{matrix}
0\\
1
\end{matrix}
\right]
.
\]
Thus, the triples  $(\ell^2, J, \{e_{-1},e_0\})$ and 
$(\ell^2,A, \{e_{-1},\tilde{e}_0\})$ are equivalent. Since such an equivalence is defined uniquely, the map $J=\mathcal F A$ is  well-defined.
\end{proof}

\begin{remark}
Since $e_{-1}$ is a common vector in the chosen cyclic subspaces for $A$ and $J=\mathcal F A$, similarly to \eqref{tildee1}, we have
\begin{equation}\label{tildee3}
b_{-1}=\langle J e_{-1}, e_{-1}\rangle=\langle A e_{-1}, e_{-1}\rangle=q_{-1}.
\end{equation}

\end{remark}

\begin{definition}
Let $\mathcal S$ be the shift on Jacobi matrices, $\mathcal S J=S^{-1}J S$, and let $\mathcal F$ be the map from SMP- to Jacobi matrices, which was defined in Theorem \ref{thm:unitriples}. We define the Jacobi flow on SMP matrices as a dynamical system
$$
A^{(n+1)}=\mathcal J A^{(n)}, \quad A^{(0)}=A,
$$
generated by the map $\mathcal J$, which makes the following diagram commutative.
\begin{equation}\label{defjfg}
\begin{array}{ccc}
 \text{SMP} & \xrightarrow{\mathcal J}  & \text{SMP}  \\
   &   &    \\
_\mathcal F  \big\downarrow &   &  _\mathcal F \big\downarrow  \\
   &   &    \\
\text{Jacobi}  &  \xrightarrow{\mathcal S}  &  \text{Jacobi} 
\end{array}
\end{equation}

\end{definition}

\subsection{Basic properties of the Jacobi flow}

We will now proceed with further investigations on the transformation $\mathcal J$. Theorem~\ref{thm:jacobiflowper} 
gives a good hint for finding how $\mathcal J$ acts on SMP matrices in the generic case in terms of explicit formulas.

\begin{theorem}
\label{def:unitary}
 Let $A$ be an SMP matrix. We define the unitary block diagonal matrix $\mathcal U(A)$ by
\begin{equation*}
 \mathcal U(A)\begin{bmatrix} e_{2m}&e_{2m+1} \end{bmatrix}=
 \begin{bmatrix} e_{2m}&e_{2m+1} \end{bmatrix} U_m(A)
\end{equation*}
where the $U_m(A)$, $m\in\mathbb Z$, are unitary $2\times2$ matrices given by
\begin{equation*}
 U_m(A)=\frac{1}{\sqrt{p_{2m}^2+r_{2m+1}^2}}\begin{bmatrix}
                                          p_{2m} & r_{2m+1} \\ r_{2m+1} & -p_{2m}
                                         \end{bmatrix}.
\end{equation*}
 The Jacobi flow transformation $\mathcal J$ maps the SMP class onto itself by
\begin{equation*}
 \mathcal J(A)=S^{-1} \mathcal U(A)^{*}A~\mathcal U(A)~S
\end{equation*}
for all SMP matrices $A$.
\end{theorem}

 In this subsection we will use the notation
\begin{equation}\label{j}
 j(A)=S^{-1} \mathcal U(A)^{*}A~\mathcal U(A)~S
\end{equation}
Our goal is to show that $\mathcal{J}(A)=j(A)$.
As usual,  $A^{(k)}=j^{\circ k}(A)$ and the same upper index is related to the coefficients of $A^{(k)}$.

 Note that, according to the definition of $\mathcal U(A)$,
\begin{equation}\label{j1}
 U^{*}_{k-1}\begin{bmatrix} 0 & 0 \\ p_{2k} & r_{2k+1}\end{bmatrix}U_k=\begin{bmatrix} r_{2k-1}^{(1)} & 0 \\ p_{2k-1}^{(1)} & 0\end{bmatrix}
 \end{equation}
 and 
 \begin{equation*}
 U^{*}_{k}\begin{bmatrix} q_{2n} & p_{2k+1} \\ p_{2k+1} & q_{2k+1}\end{bmatrix}U_k=\begin{bmatrix} q_{2k-1}^{(1)} & p_{2k}^{(1)} \\ p_{2k}^{(1)} & q_{2k}^{(1)}\end{bmatrix}.
\end{equation*}
That is, multiplying $A$ by $\mathcal U(A)$ and its conjugate, we obtain an SMP matrix shifted by one in the positive direction. 
\begin{lemma}
The unitary matrices $U_m(A)$ are defined   uniquely by the property that $j(A)$ is an SMP matrix, if, in addition, we assume that the off-diagonal entries of $U_m(A)$ are positive.
\end{lemma}

\begin{proof}
 Indeed, as
\begin{equation*}
 \begin{bmatrix} 0 & 0 \\ p_{2k} & r_{2k+1}\end{bmatrix}U_k\begin{bmatrix}0 \\  1\end{bmatrix}=0,
 \end{equation*}
the second column of $U_k$ should be orthogonal to $[p_{2k}\, r_{2k+1}]$. Since this matrix is unitary, the first column should be orthogonal to the second one.  Thus, both columns are defined uniquely up to unimodular scalar multipliers, which are defined by the normalization condition.
\end{proof}

\begin{lemma} 
The transformation \eqref{j} is invertible. 
\end{lemma}
\begin{proof}
This can be obtained easily from \eqref{j1}, but we want show this also on a ``macro-level", that is, directly from \eqref{j}. Recall that the following involutionon acts on SMP matrices
$$
\tau A=-S^{-1}A^{-1}S.
$$
Let $A^{(1)}=j(A)$. Applying $\tau$, we get
$$
\tau A^{(1)}=S^{-2}\mathcal U(A)^*S\tau A S^{-1}\mathcal U(A) S^2.
$$
That is,
$$
\tau A=S^{-1}\mathcal U(A)S^2\tau A^{(1)}S^{-2}\mathcal U(A)^*S.
$$
Since $\tau A$ and $S^2\tau A^{(1)}S^{-2}$ are SMP matrices, we have, by the uniqueness property of the $j$-transform, that
$$
\tau A=j(S^2\tau A^{(1)}S^{-2})\quad\text{and}\quad \mathcal U(A)= \mathcal{U}(S^2\tau A^{(1)}S^{-2})^*.
$$
Thus, $A^{(1)}=j(A)$ can be solved for $A$.
\end{proof}

Using the $j$-transform we can explicitly find an orthonormal basis in $\ell^2$ with respect to which $A$ becomes a Jacobi matrix. 

\begin{thrmekt}
 \label{thrm:ekt}
Let $A$ be an SMP matrix. Let
\begin{equation*}
 \tilde e_{k}=\tilde e_{k,A}=\mathcal U(A^{(0)})S \mathcal U( A^{(1)})S\cdots\mathcal U( A^{(k)})Se_{-1},\quad\forall k\geq0,
\end{equation*}
$\tilde e_{-1}=e_{-1}$, and
\begin{equation*}
 \tilde e_{k-1}=\tilde e_{k-1,A}=S^{-1} \mathcal U(A^{(-1)})\cdots S^{-1}\mathcal U( A^{(k)})^{-1} e_{-1},\quad\forall k<0.
\end{equation*}
These vectors form  an orthonormal system in $\ell^2$, with respect to which the following three term recurrence relation holds
\begin{equation}\label{jsmpcoef0}
 A\tilde e_{k-1} = a_{k-1} \tilde e_{k-2} + b_{k-1} \tilde e_{k-1}+a_{k}\tilde e_{k}, 
\end{equation}
where
\begin{equation}\label{jsmpcoef}
a_k={\sqrt{(p^{(k)}_0)^2+(r^{(k)}_1)^2}},\quad b_{k-1}=q^{(k)}_{-1}.
\end{equation}
\end{thrmekt}

 \begin{proof}
We note that $\tilde e_{-1}=e_{-1}$. Therefore,
 $$
 A\tilde e_{-1}=(r_{-1}e_{-3}+p_{-1}e_{-2})+q_{-1}e_{-1}+(p_0 e_0+r_1e_1).
 $$
 By definition we have
 $$
 p_0 e_0+r_1e_1=a_0 \tilde e_0, \quad a_0=\sqrt{p_0^2+r_1^2},
 $$
 and by  \eqref{j1}
\begin{equation*}
\begin{bmatrix}
 r_{-1}^{(0)} \\ p_{-1}^{(0)}
\end{bmatrix}
=\frac{\sqrt{(p^{(-1)}_0)^2+(r^{(-1)}_1)^2}}{\sqrt{(p^{(-1)}_{-2})^2+(r^{(-1)}_{-1})^2}} 
\begin{bmatrix}
 r^{(-1)}_{-1} \\- p^{(-1)}_{-2}
\end{bmatrix}.
\end{equation*}
Therefore,
\begin{equation*}
a_{-1}={\sqrt{(p^{(-1)}_0)^2+(r^{(-1)}_1)^2}}={\sqrt{(p^{(0)}_{-1})^2+(r^{(0)}_{-1})^2}},
\end{equation*}
and, consequently,
$$
r_{-1}e_{-3}+p_{-1}e_{-2}=a_{-1}S^{-1}\frac{r^{(-1)}_{-1}e_{-2}-p^{(-1)}_{-2}e_{-1}}{\sqrt{(p^{(-1)}_{-2})^2+(r^{(-1)}_{-1})^2}}=a_{-1}S^{-1}\mathcal U(A^{(-1)})e_{-1}.
$$
Thus, we have \eqref{jsmpcoef0} and \eqref{jsmpcoef} for $k=0$.

Now, we write the same relation for $A^{(1)}$. Using the unitary equivalence \eqref{j}, we obtain \eqref{jsmpcoef0} for $k=1$, and so on...

\end{proof}

\begin{proof}[Proof of Theorem \ref{def:unitary}]
Due to Proposition \ref{cor:cyclic} $e_{-1}$ and $\tilde e_0$ form a cyclic subspace in the entire $\ell^2$ space. By \eqref{jsmpcoef0} they form a cyclic subspace in span$\{\tilde e_k\}_{k\in\Z}$. Therefore, these spaces coincide.

Let us define $\mathbf U:\ell^2\to\ell^2$ by
$\mathbf U \tilde e_k=e_k$. Since the vectors $\tilde e_k$ are orthonormal, this map is an isometry from span$\{\tilde e_k\}_{k\in\Z}$ onto $\ell^2$. Since $\mathbf U$ is densely defined, it is unitary.  Moreover, by \eqref{jsmpcoef0},
$
\mathbf U A=J \mathbf U,
$
where $J$ is the Jacobi matrix with coefficient sequences $\{a_k, b_{k-1}\}$. Thus, the triples
$\{\ell^2, A, \{e_{-1},\tilde e_0\}\}$ and $\{\ell^2, J, \{e_{-1}, e_0\}\}$ are equivalent, or, in other words, $J=\mathcal F(A)$.

By definition we have the following relation between the bases $\{\tilde e_{k,A}\}_{k\in\Z}$ and $\{\tilde e_{k,j(A)}\}_{k\in\Z}$
$$
\tilde e_{k,A}= \mathcal U(A)S\tilde e_{k-1,j(A)}.
$$
Therefore, $\mathcal F(j(A))=S^{-1}J S$, that is, $j(A)=\mathcal J(A)$.
\end{proof}

\subsection{Main Lemma}

Recall that the Killip-Simon functional on SMP matrices has the following form:
\begin{equation}\label{eq1}
H(A)=\frac 1 2 \tr(v^{(0)})^2+\sum_j\{(v_j^{(1)})^2-1-\log (v_j^{(1)})^2\},
\end{equation}
where
$$
V(A)=S^{-2}v^{(1)}+v^{(0)}+v^{(1)}S^2, \quad v^{(0)}=S^{-1}v^{(0,1)}+v^{(0,0)}+v^{(0,1)}S.
$$
Note that
$$
v^{(1)}=\diag\{\begin{matrix}\dots& \frak ar_{-1}& -\rho_0&\frak ar_1&-\rho_2&\dots \end{matrix}\}.
$$
We will work with the half-axis functional
\begin{equation}\label{eq2}
\begin{split}
H_+(A)&=\frac 1 2 [(v_0^{(1)})^2+(v_1^{(1)})^2-2-\log (v_0^{(1)})^2-\log (v_1^{(1)})^2]+
\frac 1 2 (v_0^{(0,1)})^2\\
&+\sum_{j\ge 2}\{(v_j^{(1)})^2-1-\log (v_j^{(1)})^2\}
+\sum_{j\ge 1}(v_j^{(0,1)})^2
+\frac 1 2 \sum_{j\ge 0}(v_j^{(0,0)})^2.
\end{split}
\end{equation}

\begin{remark}
The functional \eqref{eq1} is related to the formal "trace" (the sum of all diagonal entries) of the operator 
$V(A)^2$. In its turn, \eqref{eq2} is related to its restriction to the positive half-axis, that is, to the operator 
$P_+ V(A)^2P_+$, where $P_+$ is the orthogonal projector from $\ell^2$ onto $\ell^2_+$.
\end{remark}

 In what follows, it is convenient to rewrite \eqref{eq2} into another form. Let us make the block decomposition of $V(A)$ in $2\times 2$ blocks
 \begin{equation}\label{eq11}
 V(A)=\begin{bmatrix}\ddots&\ddots&\ddots & & &\\
 &\fv^*_{-1}&\fw_{-1}&\fv_0 & & \\
  & &\fv^*_{0}&\fw_{0}&\fv_1 & \\
& &  & \ddots&\ddots\ddots
 \end{bmatrix},
 \end{equation}
 where
 $$
 \fv_k=\begin{bmatrix} v^{(1)}_{2k}&0\\ v^{(0,1)}_{2k}&v^{(1)}_{2k+1} \end{bmatrix}, \ \ 
  \fw_k=\begin{bmatrix} v^{(0,0)}_{2k}&v^{(0,1)}_{2k+1}\\ v^{(0,1)}_{2k+1}&v^{(0,0)}_{2k+1} \end{bmatrix}.
$$
Then
\begin{equation}\label{eq10}
H_+(A)=\frac 1 2\sum_{j\ge 0}\{\tr \fw_j^2+ \tr \fv_j^* \fv_j
+\tr \fv_{j+1} \fv_{j+1}^*
-4-\log\prod_{l=0}^3( v^{(1)}_{2j+l})^2\}.
\end{equation}

\begin{lemma}
Let $A^{(1)}=\cJ A$, where $\cJ$ is the Jacobi flow on SMP matrices. Define
\begin{equation}\label{eq5bis}
\begin{split}
\delta_J H_+(A)=&\frac{1} 2\{\langle V(\cJ A)e_{-1}, V(\cJ A)e_{-1}\rangle-2-\log((\cJ v)^{(1)}_{-1})^2((\cJ v)^{(1)}_1)^2\}\\
=&\frac 1 2 \{(a r^{(1)}_{-1})^2+(a r^{(1)}_{1})^2-2-\log(a r^{(1)}_{-1})^2(a r^{(1)}_1)^2\}\\
+&
\frac{(ap^{(1)}_{-1}-\pi^{(1)}_{-1})^2+(ap^{(1)}_{0}-\pi^{(1)}_{0})^2}{2}
+\frac{(aq^{(1)}_{-1}+b-\sigma^{(1)}_{-1})^2}{2}.
\end{split}
\end{equation}
Then
\begin{equation}\label{eq5}
H_+(A)=H_+(\cJ A)+\delta_J H_+(A).
\end{equation}
\end{lemma}

\begin{proof}
 Recall that $\cJ A=S^{-1}\mathcal U(A)^{-1}A\mathcal U(A)S$. So, in the sum \eqref{eq5}
we make a restriction of $V(\cJ A)^2$ not to $\ell^2_+$, but to its one dimensional extension, that is, to
  span$\{e_n\}_{n\ge -1}=S^{-1}\ell_+^2$. In other words, the functional
 $$
\Lambda:= H_+(\cJ A)+\frac{1} 2\{\langle V(\cJ A)e_{-1}, V(\cJ A)e_{-1}\rangle-2-
\log((\cJ v)^{(1)}_{-1})^2((\cJ v)^{(1)}_1)^2\}
 $$
 deals with the formal trace of the operator 
 $$
  P_+SV(\cJ A)^2S^{-1}P_+=
 P_+\mathcal U(A)^{-1}V(A)^2\mathcal U(A)P_+.
 $$
 
Using the decomposition \eqref{eq11} we rewrite $\Lambda$ as
\begin{equation*}
\begin{split}
\Lambda=&\frac 1 2\sum_{j\ge 0} \tr (U_{j}^{-1}\fw_j^2 U_j)\\
+&\frac 1 2\sum_{j\ge 0}\{  \tr(U_j^{-1}\fv_j^* \fv_j U_j)
+\tr(U^{-1}_j\fv_{j+1} \fv_{j+1}^* U_j)
-4-\log\prod_{l=-1}^2((\cJ v)^{(1)}_{2j+l})^2\}\\
=&\frac 1 2\sum_{j\ge 0}\{\tr \fw_j^2+ \tr \fv_j^* \fv_j
+\tr \fv_{j+1} \fv_{j+1}^*
-4-\log\prod_{l=-1}^2((\cJ v)^{(1)}_{2j+l})^2\}.
\end{split}
\end{equation*}

It remains to show that
\begin{equation}\label{eq6}
(\cJ v)^{(1)}_{2j-1}(\cJ v)^{(1)}_{2j}=v^{(1)}_{2j}v^{(1)}_{2j+1},\quad \text{i.e.,}\quad
r^{(1)}_{2n-1}\rho_{2n}^{(1)}=\rho_{2n}r_{2n+1}.
\end{equation}
Recall that
\begin{equation}\label{eq4}
r_{2n+1}^{(1)}=r_{2n+1}\sqrt{\frac{p_{2n+2}^2+r_{2n+3}^2}{{p_{2n}^2+r_{2n+1}^2}}}.
\end{equation}
To find a similar relation between the $\rho_k$'s and  the $\rho^{(1)}_k$'s we use
$$
(A^{(1)})^{-1}=S^{-1} \mathcal U(A)^{-1}A^{-1} \mathcal U(A) S.
$$
This implies
$$
\rho_0^{(1)}=\langle A^{-1} \mathcal U(A)e_1,\mathcal U(A) e_{-1}\rangle=\frac{
\begin{bmatrix} r_{-1}&-p_{-2}\end{bmatrix}
\begin{bmatrix}
\rho_0& 0\\ \pi_0& 0
\end{bmatrix} \begin{bmatrix} r_1\\ -p_0\end{bmatrix}}{\sqrt{p_0^2+r_1^2}\sqrt{p_{-2}^2+r_{-1}^2}}
$$
$$
=\frac{r_1(\rho_0 r_{-1}-\pi_0 p_{-2})}{\sqrt{p_0^2+r_1^2}\sqrt{p_{-2}^2+r_{-1}^2}}
=\frac{r_1\rho_0(r_{-1}^2+p_{-2}^2)}{r_{-1}\sqrt{p_0^2+r_1^2}\sqrt{p_{-2}^2+r_{-1}^2}}
=\frac{r_1\sqrt{p_{-2}^2+r_{-1}^2}}{r_{-1}\sqrt{p_0^2+r_1^2}}\rho_0.
$$
Thus, for arbitrary $n$,
\begin{equation}\label{eq7}
\rho_{2n+2}^{(1)}=\frac{r_{2n+3}\sqrt{p_{2n}^2+r_{2n+1}^2}}{r_{2n+1}\sqrt{p_{2n+2}^2+r_{2n+3}^2}}\rho_{2n+2}.
\end{equation}
Combining \eqref{eq4} and \eqref{eq7} we get \eqref{eq6}. Therefore, due to \eqref{eq10}, we obtain 
$\Lambda=H_+(A)$, indeed.
\end{proof}


\begin{thebibliography}{99}
 \bibitem{AKHef} N.I. Akhiezer, \textit{Elements of the Theory of Elliptic Functions}. Amer. Math. Soc., Providence, RI, 1990.

 \bibitem{kssmp}
R. Ensgraber, F. Puchhammer, P. Yuditskii, \emph{Killip-Simon-classes of Jacobi matrices with essential spectrum on two symmetric 
and of SMP matrices on two arbitrary intervals}, arXiv:1309.0959v2 [math.SP].

\bibitem{JN}
W. B. Jones and O. Nj{\aa}stad, \emph{Orthogonal Laurent polynomials and strong moment theory: a survey.}
Continued fractions and geometric function theory (COFUN) (Trondheim 1997).
 J. Comput. Appl. Math., 105 (1999), no. 1-2, 51-91.

 \bibitem{ionela}
 I. Moale, P. Yuditskii, \emph{Parametrization of spectral surfaces of a class of periodic 5-diagonal matrices}, arXiv:1112.3219 [math.SP].

\bibitem{SimOPUC}
B. Simon, \textit{Orthogonal polynomials on the unit circle}. Part 2. Spectral theory. American Mathematical Society Colloquium Publications, 54, Part 2.


\bibitem{Sim1} K. Simonov, \textit{Orthogonal Matrix Laurent Polynomials},
Mathematical Notes, vol. 79, no. 2, 2006, pp. 291-295.

\bibitem{Sim2} K. Simonov, \textit{Strong matrix moment problem of Hamburger}, Methods of Functional Analysis and Topology Vol. 12 (2006), no. 2, pp. 183-196.


\bibitem{sodyud}
M. Sodin, P. Yuditskii \emph{Almost Periodic Jacobi Matrices with Homogeneous Spectrum, Infinite Dimensional Jacobi Inversion,
and Hardy Spaces of Character-Automorphic Functions}, Journ. of Geom. Analysis 7 (1997), pp. 387-435.
\end{thebibliography}
\end{document}